\documentclass[12pt]{article}
\usepackage{amsmath,amssymb,amsthm}
\usepackage[margin=1in]{geometry}
\usepackage{cite}
\usepackage{hyperref}

\newtheorem{theorem}{Theorem}[section]
\newtheorem{lemma}[theorem]{Lemma}
\newtheorem{corollary}[theorem]{Corollary}

\newtheorem{claim}[theorem]{Claim}

\newtheorem{conjecture}[theorem]{Conjecture}

\newtheorem{observation}[theorem]{Observation}

\newtheorem{remark}[theorem]{Remark}

\title{A sharp extension of Halin’s removable-edge theorem to matchings}

\author{Hojin Chu\thanks{Supported by a KIAS Individual Grant (CG101801) at
Korea Institute for Advanced Study.}\\
\small School of Computational Sciences, Korea Institute for Advanced Study
(KIAS), Seoul, Korea\\
\small \texttt{hojinchu@kias.re.kr}}

\date{\today}

\begin{document}
\maketitle

\begin{abstract}
A subgraph $H$ of a $k$-connected graph $G$ is called \emph{$k$-removable} if $G-E(H)$ remains $k$-connected.
Halin proved that every $k$-connected graph $G$ with $\delta(G)\ge k+1$ has a $k$-removable edge. 
We extend this result from a single edge to matchings of any prescribed size by showing that, for positive integers $k$ and $m$, every $k$-connected graph $G$ with $\delta(G)\ge\max\{k+1,2m-2\}$ contains a $k$-removable matching of size $m$, unless $G\cong K_{2m-1}$, or $(k,m)=(1,2)$ and $G$ is a cycle. 
This confirms a conjecture of Li, Zhou, Fujita, and Mao. 
The minimum degree bound is sharp, and both exceptions are unavoidable.
Consequently, $\max\{k+1,2m-1\}$ is the sharp minimum degree threshold guaranteeing such a matching without exceptions.
The proof combines a prescribed-set strengthening of Halin's removable-edge theorem with an extremal analysis of maximum $k$-removable matchings.
\end{abstract}

\section{Introduction}
All graphs considered in this paper are finite, undirected, and simple.
For a graph $G$, let $V(G)$, $E(G)$, $\delta(G)$, and $\kappa(G)$ denote its vertex set, edge set, minimum degree, and connectivity, respectively.
For $X\subseteq V(G)$, let $G[X]$ denote the subgraph of $G$ induced by $X$, and write $G-X=G[V(G)\setminus X]$.
For $F\subseteq E(G)$, let $G-F$ denote the graph obtained from $G$ by deleting the edges in $F$.

A central problem in connectivity-keeping subgraph deletion is to determine minimum degree conditions under which prescribed subgraphs can be deleted while preserving connectivity.
In the edge-deletion setting, a subgraph $H$ of a $k$-connected graph $G$ is called \emph{$k$-removable} if $G-E(H)$ remains $k$-connected.
We focus on the case where $H$ is a matching; such an $H$ is called a \emph{$k$-removable matching}.
For a recent survey, see~\cite{TianMeng26}; for results on vertex-deletion, see~\cite{CKL72, Mader10, Mader12, HongLiu21}.

No edge of a $k$-regular $k$-connected graph is $k$-removable.
Consequently, a minimum degree condition of at least $k+1$ is necessary for the existence of a $k$-removable edge.
Halin~\cite{Halin69} proved that this condition is also sufficient.

\begin{theorem}[Halin~\cite{Halin69}]\label{thm:Halin}
    Let $k$ be a positive integer. Every $k$-connected graph $G$ with $\delta(G)\ge k+1$ contains a $k$-removable edge.
\end{theorem}

Hasunuma~\cite{Hasunuma23} proposed the following edge-deletion analogue of Mader's conjecture: for every tree $T$ of order $m$, every $k$-connected graph $G$ with $\delta(G)\ge k+m-1$ contains a copy $T'\cong T$ such that $G-E(T')$ remains $k$-connected.
Hasunuma proved the conjecture for $k\le 2$. The case $k=3$ was subsequently proved by Liu, Liu, and Hong~\cite{LLH23} and independently by Yang and Tian~\cite{YangTian24}.
Clay and Jord\'{a}n~\cite{CJ26+} recently proved both the conjecture and its edge-connectivity analogue.

Li, Zhou, Fujita, and Mao~\cite{LZFM26+} asked whether several pairwise disjoint edges can be removed simultaneously while preserving $k$-connectivity.
They extended Theorem~\ref{thm:Halin} from a single edge to a matching of size two by showing that every $k$-connected graph $G$ with $\delta(G)\ge k+1$ contains a $k$-removable matching of size two, unless $k=1$ and $G$ is a cycle.
They also obtained larger removable matchings for general $k$ under stronger minimum degree assumptions.
Motivated by these results, they proposed the following conjecture, which we state with the necessary cycle exception when $k=1$.

\begin{conjecture}\label{conj:matching}
    Let $k$ be a positive integer.
    Every $k$-connected graph $G$ with $\delta(G)\ge k+1$ contains a $k$-removable matching of size $\left\lceil(\delta(G)+1)/2\right\rceil$,
    unless either $k=1$ and $G$ is a cycle, or $\delta(G)$ is even and
    $G\cong K_{\delta(G)+1}$.
\end{conjecture}

Chu, Kim, and Park~\cite{CKP26+} proved that every $k$-connected graph $G$ on at least $2m$ vertices has a $k$-removable matching of size $m$ if
\[
\delta(G)\ge
\begin{cases}
\max\{k+\lceil m/2\rceil,2m\}, & \text{if } k\ge m,\\
k+m, & \text{if } k<m.
\end{cases}
\]
In particular, their result verifies
Conjecture~\ref{conj:matching} when $\delta(G)\ge2k+1$.
Clay and Jord\'{a}n~\cite{CJ26+} subsequently proved that every $k$-edge-connected graph $G$ with at least $2m$ vertices and
$\delta(G)\ge k+m$ contains a $k$-removable matching of size $m$.

Our main result determines the sharp minimum degree condition for the existence of a $k$-removable matching of any prescribed size, including all exceptions at the boundary.

\begin{theorem}\label{thm:main}
    Let $k$ and $m$ be positive integers.
    Every $k$-connected graph $G$ with
    \[
        \delta(G)\ge \max\{k+1,2m-2\}
    \]
    contains a $k$-removable matching of size $m$, unless $G\cong K_{2m-1}$, or $(k,m)=(1,2)$ and $G$ is a cycle.
\end{theorem}

Both exceptions in Theorem~\ref{thm:main} are unavoidable, since $K_{2m-1}$ has no matching of size $m$ and a cycle has no $1$-removable matching of size two.
Consequently, the condition $\delta(G)\ge \max\{k+1,2m-1\}$ is the sharp minimum degree condition guaranteeing a $k$-removable matching of size $m$ without exceptions.
Taking $m=\left\lceil(\delta(G)+1)/2\right\rceil$ in Theorem~\ref{thm:main} resolves Conjecture~\ref{conj:matching}.

\begin{corollary}\label{cor:large-removable-matching}
    Let $k$ be a positive integer. Every $k$-connected graph $G$ with $\delta(G)\ge k+1$ contains a $k$-removable matching of size $\left\lceil(\delta(G)+1)/2\right\rceil$, unless either $k=1$ and $G$ is a cycle, or $\delta(G)$ is even and $G\cong K_{\delta(G)+1}$.
\end{corollary}

The proof of Theorem~\ref{thm:main} relies on a prescribed-set strengthening of Halin's theorem. A different sufficient condition for such an edge was proved in~\cite{CKP26+}.
Theorem~\ref{thm:prescribed-set-removable-edge} characterizes the exceptional configuration at the boundary threshold $\max\{k+1,|W|\}$. In the proof of the main theorem, $W$ is the set of vertices covered by a maximum $k$-removable matching. 
We record the following exception-free consequence, which reduces to Theorem~\ref{thm:Halin} when $W$ is the empty set.

\begin{corollary}\label{cor:prescribed-set-removable-edge}
    Let $G$ be a $k$-connected graph, and let $W\subsetneq V(G)$.
    Set $U=V(G)\setminus W$ and $w=|W|$.
    Then $G[U]$ contains a $k$-removable edge if
        \[
        d_G(x)\ge \max\{k+1,w+1\}
    \]
    for every $x\in U$.
\end{corollary}

The rest of the paper is organized as follows. Section~2 proves the
prescribed-set strengthening of Halin's theorem, Section~3 establishes
the two dense matching lemmas needed for the extremal configurations,
and Section~4 proves Theorem~\ref{thm:main}.

\section{Removable edges outside a prescribed set}

An edge $e$ of a $k$-connected graph $G$ is called \emph{$k$-essential} if $G-e$ is not $k$-connected. We use the following theorem of Mader~\cite[Satz~1]{Mader72}. An English formulation can also be found in the survey of Kriesell~\cite[Theorem~2.1]{Kriesell13}.

\begin{theorem}[Mader~\cite{Mader72}]\label{thm:mader}
Let $G$ be a $k$-connected graph, and let $C$ be a cycle all of whose edges are $k$-essential in $G$. Then $C$ contains a vertex of degree $k$ in $G$.
\end{theorem}

We prove a strengthening of Theorem~\ref{thm:Halin} that guarantees a $k$-removable edge with both ends outside a prescribed set of vertices.

\begin{theorem}\label{thm:prescribed-set-removable-edge}
    Let $G$ be a $k$-connected graph, and let $W\subseteq V(G)$. 
Set $U=V(G)\setminus W$ and $w=|W|$. 
Suppose that $E(G[U])\neq\emptyset$ and $d_G(x)\ge \max\{k+1,w\}$ for every $x\in U$.
Then $G[U]$ contains a $k$-removable edge, unless $w=k+1$ and
$G[U]$ is a forest such that, for every leaf $u$, there exists $z_u\in W$ satisfying
\[
    N_G(u)=(W\setminus\{z_u\})\cup\{v_u\},
\]
where $v_u$ is the unique neighbor of $u$ in $U$.
\end{theorem}
\begin{proof}
Suppose that no edge of $G[U]$ is $k$-removable in $G$. Thus every edge of $G[U]$ is $k$-essential in $G$.
If $G[U]$ contained a cycle, then Theorem~\ref{thm:mader} implies that the cycle contains a vertex of degree $k$ in $G$, which contradicts $d_G(x)\ge k+1$ for every $x \in U$.
Therefore $G[U]$ is a forest.

Take a leaf $u$ in $G[U]$, and let $v$ be its unique neighbor in $U$. Since $G[U]$ has an edge, $u$ exists.
Note that $N_G(u)\subseteq W\cup\{v\}$.
Thus we have
\[
    \max\{k+1,w\}\le d_G(u)\le w+1.
\]
Therefore $N_G(u)=W\cup \{v\}$ or $N_G(u)=(W\setminus\{z\})\cup\{v\}$ for some $z\in W$.

Let $a=|W\setminus N_G(u)|$. Then $a\le 1$. 
We will show that $a=1$. 
Since $uv$ is $k$-essential in $G$, $G-uv$ has a vertex cut $S$ with $|S|\le k-1$. 
Since $G-S$ is connected, neither $u$ nor $v$ belongs to $S$. Moreover, $G-uv-S$ has exactly two components $X$ and $Y$ joined in $G-S$ only by $uv$, where $u\in X$ and $v \in Y$.
Then no vertex of $W\cap Y$ is adjacent to $u$, so $|W\cap Y|\le a$.
Let $A=U\cap Y$.
Since $G[U]$ is a forest, $G[A]$ is a forest.
Then we may choose a vertex $y$ of $G[A]$ such that $d_{G[A]}(y)\le 1$.
In particular, if $|A|\ge 2$, we choose $y \neq v$.
If $y=v$, then $A=\{v\}$ and every neighbor of $v$ other than $u$ belongs to $S\cup (W\cap Y)$.
If $y\neq v$, then $y$ has no neighbor in $X$.
Consequently, in either case 
\[
d_G(y)\le |S|+|W\cap Y|+1 \le k+a.
\]
On the other hand, $y\in U$ and so $d_G(y)\ge \max\{k+1,w\}$.
Thus $a\ge 1$ and $k\ge w-1$.
Then $a=1$ since $a\le 1$.
Moreover, $k=w-1$ since $k+1\le d_G(u)=|W|+1-a=w$.
Hence 
\[
N_G(u)=(W\setminus\{z\})\cup\{v\}
\] for some $z\in W$.
\end{proof}

\begin{remark}
    The exceptional case in Theorem~\ref{thm:prescribed-set-removable-edge} occurs only when $w=k+1$ and all the leaves in a forest $G[U]$ have degree $w$.
    If a vertex in $U$ has degree $w+1$, then there is an edge outside $W$.
    Thus strengthening the degree condition to $d_G(x)\ge \max\{k+1,w+1\}$ eliminates the exceptional case and immediately yields Corollary~\ref{cor:prescribed-set-removable-edge}.
\end{remark}

\section{Auxiliary results on removable matchings}

We first recall the following result of Li, Zhou, Fujita, and Mao, which settles the case $k=1$.

\begin{lemma}[Li, Zhou, Fujita, and Mao~\cite{LZFM26+}]\label{lem:1-removable 2-matching}
    Every connected graph $G$ with $\delta(G) \ge 2$ has a $1$-removable matching of size $2$, unless $G$ is a cycle.
\end{lemma}

We also use the Tutte--Berge formula for the size of a maximum matching in a graph.

\begin{theorem}[Tutte--Berge formula~\cite{Berge}]
Let $G$ be a graph of order $n$.
Then the maximum size of a matching in $G$ is equal to 
\[
\frac{1}{2}\left(n-\max_{S\subseteq V(G)}(o(G-S)-|S|)\right), 
\]
where $o(G-S)$ is the number of odd components in $G-S$.
\end{theorem}

We next establish two lemmas in dense graphs, corresponding to the two final cases in the proof of Theorem~\ref{thm:main}.

\begin{lemma}\label{lem:dense-matching}
Let $G$ be a graph of order $n\ge 6$. If $\delta(G)\ge n-2$, then $G$ has a matching $P$ of size $\lfloor n/2 \rfloor$ such that $G-P$ is $(n-3)$-connected.
\end{lemma}

\begin{proof}
Suppose $\delta(G)\ge n-2$.
Then the complement $\overline{G}$ satisfies $\Delta(\overline{G}) \le 1$, so $E(\overline{G})$ is a matching.
Therefore we may extend it to a matching $Q=\{x_i y_i:1\le i\le \lfloor n/2 \rfloor\}$ of size $\lfloor n/2 \rfloor$ in the complete graph on $V(G)$.
Let
\[
P=\{y_i x_{i+1}:1\le i\le \lfloor n/2 \rfloor\},
\] 
where the subscripts are taken modulo $\lfloor n/2 \rfloor$.
Then $P$ is a matching in $G$.
Moreover, $Q\cup P$ is a cycle of length $2\lfloor n/2 \rfloor$.

We claim that $P$ is a desired matching. 
Let $H=G-P$. 
Suppose to the contrary that $H$ is not $(n-3)$-connected.
Then $H$ has a vertex cut $S$ with $|S|\le n-4$.
Take a component $X$ of $H-S$ and let $Y=V(H)\setminus(S\cup X)$.
For any $x\in X$ and $y\in Y$, the edge $xy$ belongs to $\overline{H}$, so 
\[
xy\in E(\overline{H})= E(\overline{G})\cup P\subseteq Q\cup P.
\] 
Thus $Q\cup P$ contains the complete bipartite graph with parts $X$ and $Y$.
Since $Q\cup P$ is a cycle, $|X|\le2$ and $|Y|\le2$.
On the other hand, $|X|+|Y|=n-|S|\ge4$. 
Hence, $|X|=|Y|=2$ and it follows that $Q \cup P$ has a $4$-cycle as a subgraph.
This is impossible since $Q \cup P$ is a cycle of length $2\lfloor n/2 \rfloor\ge 6$.
Therefore $H$ is $(n-3)$-connected.
\end{proof}

\begin{observation}\label{obsv:K_2,3}
    Let $G$ be a graph of order $n\ge 5$.
    Suppose $\delta(G)\ge n-4$.
    Then $G$ is $(n-4)$-connected if and only if $\overline{G}$ is $K_{2,3}$-free.
\end{observation}
\begin{proof}
    If $\overline{G}$ has a copy of $K_{2,3}$, then the other $n-5$ vertices form a vertex cut, so $G$ is not $(n-4)$-connected.
    For the converse, suppose that $G$ has a vertex cut $S$ with $|S|\le n-5$.
    Let $A$ and $B$ be a partition of $V(G)\setminus S$ such that $G-S$ has no edge joining $A$ and $B$.
    Then $\overline{G}$ has a complete bipartite graph with parts $A$ and $B$.
    Since $\delta(G)\ge n-4$, both $A$ and $B$ have size at least $n-3-|S|$.
    If $|S|\le n-6$, then both $A$ and $B$ have size at least $3$.
    If $|S|=n-5$, then one of $A$ and $B$ has size $2$ and the other has $3$.
    Thus, in either case, $\overline{G}$ has a copy of $K_{2,3}$.
\end{proof}

\begin{lemma}\label{lem:n-3}
    Let $G$ be a graph of order $n\ge 7$.
    If $\delta(G)\ge n-3$, then $G$ has a matching $P$ of size $\lfloor n/2 \rfloor$ such that $G-P$ is $(n-4)$-connected.
\end{lemma}

\begin{proof}
Suppose $\delta(G)\ge n-3$. We define an auxiliary graph $Q$ on $V(G)$ by joining two vertices $u$ and $v$ if $uv\in E(G)$ and $u,v$ are not the end vertices of a path of length three in $\overline{G}$. 
Since $\delta(G)\ge n-3$, $\Delta(\overline{G})\le 2$, so each vertex is an end vertex of a path of length three with at most two other vertices.
Hence $\delta(Q)\ge \delta(G)-2=n-5$.

We claim that $G-P$ is $(n-4)$-connected for every matching $P$ of $G$ such that 
\[
|P\setminus E(Q)|\le 1.
\]
Suppose to the contrary that there exists a matching $P$ of $G$ satisfying $|P\setminus E(Q)|\le 1$ such that $G-P$ is not $(n-4)$-connected.
Since $\delta(G-P)\ge \delta(G)-1\ge n-4$, Observation~\ref{obsv:K_2,3} implies that $\overline{G-P}$ has a copy of $K_{2,3}$.
Let $\{a,b\}$ and $\{x,y,z\}$ be the bipartition of a copy of $K_{2,3}$ in $\overline{G-P}$.
Since $\Delta(\overline{G})\le 2$ and $P$ is a matching, we may assume that $ax,by\in P$. 
Since $|P\setminus E(Q)|\le1$, at least one of $ax$ and $by$ belongs to $E(Q)$.
By symmetry, assume $ax\in E(Q)$.
Then $\{az,zb,bx\}\in E(\overline{G})$, so $azbx$ is a path of length $3$ in $\overline{G}$, which contradicts $ax\in E(Q)$.

Therefore, it suffices to find a matching $P$ in $G$ of size
$\lfloor n/2\rfloor$ such that $|P\setminus E(Q)|\le1$.
Suppose $n\ge9$. Then $\delta(Q)\ge n-5\ge\lfloor n/2\rfloor$.
We define a graph $Q'$ obtained from $Q$ by adding a new vertex $v$ that is adjacent to every vertex of $Q$.
Then $\delta(Q') \ge (n+1)/2$, so $Q'$ has a Hamiltonian cycle $C$ by Dirac's theorem. 
Thus $C-v$ is a Hamiltonian path in $Q$, whose alternate edges form a matching $P$ of size $\lfloor n/2\rfloor$.
Since $P\subseteq E(Q)$, we have $|P\setminus E(Q)|=0$.

It remains to consider $n\in\{7,8\}$. Set $d=n-5$.
Then $d\in\{2,3\}$, $\lfloor n/2\rfloor=d+1$, and $\delta(Q)\ge d$.
If $Q$ has a matching of size $d+1$, then we are done.
Hence, we may assume that $Q$ has no matching of size $d+1$. 
By the Tutte--Berge formula, there exists $S\subseteq V(Q)$ such that
\[
o(Q-S)\ge n+|S|-2d= |S|-d+5.
\]
If $Q$ is disconnected, then the minimum degree condition forces every component to have at least $d+1$ vertices.
For $d=2$, the only possible component orders are $3$ and $4$, and each component has a matching of size $1$ and $2$, respectively. 
For $d=3$, the two components have order $4$ and hence they are complete.
In either case, $Q$ has a matching of size $d+1$, a contradiction. 
Thus $Q$ is connected. Since $o(Q)\le1<5-d$, $S\ne\varnothing$.
If $|S|>d$, then $o(Q-S)\ge 6>|V(Q)\setminus S|$, which is impossible. 
If $|S|<d$, then $(d,|S|)\in \{(2,1),(3,1),(3,2)\}$, so, since every odd component of $Q-S$ has at least $d-|S|+1$ vertices,
\[
n-|S| \ge o(Q-S)(d-|S|+1) \ge (|S|-d+5)(d-|S|+1) >n-|S|,
\]
a contradiction.
Therefore $|S|=d$.
Then $|V(Q)\setminus S|=o(Q-S)=5$, and so $Q-S$ consists of five isolated vertices.
Since $\delta(Q)\ge d$, every vertex in $V(Q)\setminus S$ is adjacent to every vertex of $S$ and has degree $d$ in $Q$.
Each vertex of $G-S$ has precisely $d$ neighbors in $Q$, all in $S$.
By the definition of $Q$ and the fact that $\delta(G)\ge d+2$, every
vertex of $G-S$ has exactly two neighbors in $G-S$. 
Therefore, both $G-S$ and $\overline{G}-S$ are $5$-cycles.
Let $v_1v_2v_3v_4v_5v_1$ be the cycle of $\overline{G}-S$, and let $S=\{x_1,\ldots,x_d\}$.
Define
\[ 
P=
\begin{cases}
\{v_1v_3,v_2x_1,v_4x_2\} & \text{if } d=2,\\
\{v_1v_3,v_2x_1,v_4x_2,v_5x_3\} & \text{if }d=3.
\end{cases}
\]
Then $P$ is a matching in $G$ of size $\lfloor n/2\rfloor$.
Moreover, all edges of $P$ except $v_1v_3$ belong to $Q$.
Hence $|P\setminus E(Q)|=1$.
\end{proof}

\section{Proof of Theorem~\ref{thm:main}}

We are now ready to prove Theorem~\ref{thm:main}.

\begin{proof}[Proof of Theorem~\ref{thm:main}]
Let $G$ be a $k$-connected graph with $\delta(G)\ge \max\{k+1,2m-2\}$.
Suppose to the contrary that $G$ is a counterexample to Theorem~\ref{thm:main}.
Then $G$ has no $k$-removable matching of size $m$. 
Moreover, $G\not\cong K_{2m-1}$, and if $(k,m)=(1,2)$, $G$ is not a cycle.
We have $|V(G)|\ge 2m$ since $G \not\cong K_{2m-1}$ and $\delta(G) \ge 2m-2$. 
Let $M$ be a maximum $k$-removable matching in $G$. 
Among all such matchings, choose $M$ so that the number of components of $G[U]$ is minimum, where 
\[
W=V(M) \qquad \text{and} \qquad U=V(G)\setminus W.
\] 
Denote $r=|M|$. Then $r\le m-1$ and $|W|=2r$.
Let $H=G-M$. 
Since every edge of $M$ has both end vertices in $W$, deleting $M$ does not affect the subgraph induced by $U$. Hence
\[
H[U]=G[U].
\]
Thus 
\[
d_H(u)=d_G(u)\ge \max\{k+1,2m-2\}
\] 
for every $u\in U$.
We use $G[U]$ throughout. 
We first claim that $G[U]$ has an edge.
\begin{claim}\label{clm:edge_exists}
     $G[U]$ has an edge.
\end{claim}
\begin{proof}
Suppose not.
Since $|U|=|V(G)|-2r\ge 2$, $U$ has distinct vertices $u$ and $v$.
As $U$ is independent, $N_H(u)\subseteq W$, so 
\[
2m-2 \le \delta(G) \le d_H(u) \le |W|=2|M|\le 2m-2. 
\]
Hence $|M|=m-1$ and $N_H(u)=W$. Similarly, $N_H(v)=W$. 
The matching $M$ is nonempty since otherwise $G[U]=G$ has an edge.
Choose an edge $xy \in M$. 
Then the graph $H+xy$ is $k$-connected and has a cycle $C:=uxvyu$. 
Every vertex on $C$ has degree at least $k+1$ in $H+xy$, so by Theorem~\ref{thm:mader}, there is an edge $e$ on $C$ which is $k$-removable in $H+xy$.
Set
\[
M':=(M-\{xy\})\cup \{e\}.
\]
Then, since $G-M'=H+xy-e$, $M'$ is a $k$-removable matching of size $r$.
Furthermore, the edge on $C$ opposite to $e$ has both ends outside $V(M')$.
Hence, $G-V(M')$ has fewer components than $G[U]$, which contradicts the choice of $M$. 
Thus $E(G[U])\neq \emptyset$.    
\end{proof}

We now apply Theorem~\ref{thm:prescribed-set-removable-edge} to $H$ with prescribed set $W$.

\begin{claim}\label{clm:G[U]}
    $G[U]$ is a forest. For every leaf $u$ of $G[U]$, let $v_u$ denote the unique neighbor of $u$ in $U$. Then there exists $z_u\in W$ such that
    \[
        N_G(u)=(W\setminus\{z_u\})\cup\{v_u\}.
    \]
    Furthermore, $r=m-1$, $k=2m-3$, and $\delta(G)=2m-2$.
\end{claim}
\begin{proof}
    By the maximality of $M$, every edge of $G[U]$ is $k$-essential in $H$.
    Note that $E(G[U])\ne\emptyset$ by Claim~\ref{clm:edge_exists}, $|W|\le 2m-2$, and $d_H(u)\ge \max\{k+1,2m-2\}$ for each $u\in U$.
    Thus, by Theorem~\ref{thm:prescribed-set-removable-edge}, $2r=k+1$ and $G[U]$ is a forest.
    Moreover, for every leaf $u$ of $G[U]$, there exists $z_u\in W$ satisfying $N_G(u)=(W\setminus\{z_u\})\cup\{v_u\}$, where $v_u$ is the unique neighbor of $u$ in $U$.
    For such a leaf $u$,
\[
2m-2=\max\{k+1,2m-2\}\le d_H(u)=2r \le 2m-2. 
\]
It follows that $r=m-1$, $k=2m-3$, and $\delta(G)=2m-2$.
\end{proof}

If $k=1$, then $m=2$, so Lemma~\ref{lem:1-removable 2-matching} implies a contradiction since $G$ is not a cycle.
Hence $k\neq 1$. Since $k=2m-3$, $k\ge 3$ and $m \ge 3$.

Recall that we choose $M$ among maximum $k$-removable matchings of $G$ so that the number of components of $G[U]$ is minimum.
This choice yields a useful consequence of the replacement argument.

\begin{claim}\label{clm:ends-only}
    Let $T$ be a nontrivial component of $G[U]$, and let $p$ and $q$ be distinct leaves of $T$. If there exists an edge $xy \in M$ such that both $x$ and $y$ are adjacent to $p$ and $q$, then 
    \[
        N_G(x)\cap U= N_G(y) \cap U=\{p,q\}.
    \]
\end{claim}
\begin{proof}
    Let $P$ be the path from $p$ to $q$ in $T$.
    Fix $t\in \{x,y\}$, and let $t'$ be the other end of $xy$.
    Since $t$ was arbitrarily chosen from $\{x,y\}$, it suffices to show that $N_G(t')\cap U=\{p,q\}$.
    The graph $H+xy$ contains the cycle $C:=tpPqt$.
    By Theorem~\ref{thm:mader}, the cycle $C$ contains a $k$-removable edge $e$ in $H+xy$. Set 
    \[
    M':=(M-\{xy\})\cup\{e\}.
    \]
    Then $M'$ is another maximum $k$-removable matching of $G$. 
    By Theorem~\ref{thm:mader} and the choice of $M$, $G-V(M')$ is a forest and has at least as many components as $G[U]$.

    If $e\in E(P)$, then necessarily $T=pq$ and $e=pq$.
    Otherwise, $t$, $t'$, and at least one of $p$ and $q$ are in $G-V(M')$ and induce a triangle, a contradiction.
    Thus either $e=ts$ for some $s\in \{p,q\}$ or $T=e=pq$.

    Suppose that  $e=ts$ for some $s\in \{p,q\}$. Let $s'$ be the other vertex in $\{p,q\}$ distinct from $s$. 
    If $t'$ has a neighbor in $V(T)\setminus\{p,q\}$, then
    the neighbor and $s'$ are joined by a path in $T-s$, which together with $t'$ forms a cycle in $G-V(M')$, a contradiction.
    If $t'$ has a neighbor in another component of $G[U]$, then $G-V(M')$ has fewer components than $G[U]$, which contradicts the choice of $M$.
    Therefore $N_G(t')\cap U=\{p,q\}$.

    Suppose $T=e=pq$.
    Then both $t$ and $t'$ belong to $G-V(M')$. 
    Since $tt' \in E(G)$, any neighbor of $t'$ in $U\setminus \{p,q\}$ makes $G-V(M')$ have fewer components than $G[U]$, which contradicts the choice of $M$.
    Thus $N_G(t')\cap U=\{p,q\}$.
\end{proof}

 To characterize the structure of $G[U]$, we first show that $G[U]$ is a union of paths.

\begin{claim}\label{clm:path-components}
    Every nontrivial component of $G[U]$ is a path.
\end{claim}
\begin{proof}
    Recall that $G[U]$ is a forest, $k\ge 3$, and $r=(k+1)/2\ge 2$.
    To show that every nontrivial component of $G[U]$ is a path, suppose to the contrary that a component $T$ of $G[U]$ has distinct leaves $u_1,u_2,u_3$.
    Denote $z_i:=z_{u_i}$ for each $i\in \{1,2,3\}$.
    If some edge $ab\in M$ is disjoint from $\{z_1,z_2,z_3\}$, then both $a$ and $b$ are adjacent to all three leaves, which contradicts Claim~\ref{clm:ends-only}. 
    Hence every edge of $M$ is incident to $z_1$, $z_2$, or $z_3$. 
    Thus $r \le 3$.
    Moreover, there are distinct $i,j \in \{1,2,3\}$ and an edge $e \in M$ such that $V(e)\cap\{z_i,z_j\}=\emptyset$. 
    Then both $u_i$ and $u_j$ are adjacent to the vertices in $V(e)$.
    By Claim~\ref{clm:ends-only}, $N_G(x)\cap U=\{u_i,u_j\}$ for a vertex $x$ in $V(e)$.
    On the other hand, for $\ell\in  \{1,2,3\} \setminus\{i,j\}$, at least one of the vertices in $V(e)$ is different from $z_{\ell}$, and hence it is adjacent to $u_{\ell}$, a contradiction.
    Therefore every nontrivial component of $G[U]$ is a path.
\end{proof}

We establish an additional property for the case $k=3$.

\begin{claim}\label{clm:k3-same-edge}
Suppose $k=3$. Let $P$ be a nontrivial path of $G[U]$ with two ends $u$ and $u'$. 
Then $z_u$ and $z_{u'}$ lie on the same edge of $M$.
\end{claim}

\begin{proof}
Suppose to the contrary that $z_u$ and $z_{u'}$ lie on different edges of $M$.
Since $k=3$, by Claim~\ref{clm:G[U]}, $r=2$ and $\delta(G)=4$. 
Let $M=\{aa',bb'\}$ where $z_u=a$ and $z_{u'}=b$. 
Then
\[
N_G(u)\cap W=\{a',b,b'\} \quad \text{and} \quad  N_G(u')\cap W=\{a,a',b'\}.
\]
Denote $P:=u_0u_1\dotsm u_t$, where $u_0=u$ and $u_t=u'$.
Since no edge of $G[U]$ is $3$-removable in $H$, $H-u_0u_1$ has a vertex cut $S$ with $|S|\le 2$. 
Moreover, $u_0u_1$ is a cut edge of $H-S$ since $H-S$ is connected.
Let $X$ and $Y$ be the components of $H-u_0u_1-S$ containing $u_0$ and $u_1$, respectively.

We first show that $S\neq \{a',b'\}$. 
Suppose to the contrary that $S=\{a',b'\}$.
Then the subpath of $P$ from $u_1$ to $u'$ is contained in $Y$, so $b\in X$ and $a\in Y$.
Since $N_H(b)\cap\bigl(W\cup V(P)\bigr)\subseteq\{a',u\}
$ and $\delta(H)\ge 3$, $b$ has a neighbor in some component $Q$ of $G[U]$ different from $P$.
Then $Q$ cannot be isolated
since an isolated vertex of $G[U]$ is adjacent to every vertex of $W$, including $a\in Y$, by the minimum degree condition of $G$.
Thus $Q$ is a nontrivial path. Let $p$ and $q$ be the end vertices of $Q$.
Since $Q$ lies entirely in $X$, $z_p=z_q=a$. 
Hence both $b$ and $b'$ are adjacent to $p$ and $q$, and Claim~\ref{clm:ends-only} implies $N_G(b)\cap U=\{p,q\}$, which contradicts $bu\in E(G)$. 

If $t=1$ or $S\subseteq W$, then $u' \in Y$, so common neighbors $a'$ and $b'$ of $u$ and $u'$ must belong to $S$, a contradiction. 
Hence $t\ge 2$ and $S \cap U \neq \emptyset$.
If $S\cap W\subseteq \{a\}$, then $N_G(u)\cap W=\{a',b,b'\}\subseteq X$, so 
\[
d_G(u_1)\le |N_G(u_1)\cap U|+|N_G(u_1)\cap W| \le 2+1=3,
\]
a contradiction.
Therefore $S=\{s_U,s_W\}$ for some $s_U\in U$ and $s_W \in \{a',b,b'\}$.
Then the two vertices in $\{a',b,b'\}\setminus\{s_W\}$ are in $X$.
Since $u_1\in Y$ and $\delta(G)=4$, $N_G(u_1)\cap W=\{a,s_W\}$ and  $a \in Y$.
Note that $u'$ is adjacent to both $a$ and the vertex in $\{a',b'\}\setminus\{s_W\}$.
Therefore $s_U=u'$.

We now show that $G[U]=P$.
Let $Q$ be a component of $G[U]$ different from $P$.
Since $s_U=u'$, $Q$ lies entirely in $X$ or entirely in $Y$.
If $Q$ is isolated, then its vertex is adjacent to every vertex of $W$, which contradicts that $u_0u_1$ is a cut edge of $H-S$.
Suppose that $Q$ is nontrivial. 
Then $Q$ has two end vertices $p$ and $q$.
If $Q\subseteq X$, then  $z_p=z_q=a$ and Claim~\ref{clm:ends-only} applied to $bb'$ contradicts $bu \in E(G)$.
If $Q\subseteq Y$, then each of $p$ and $q$ is not adjacent to the two vertices of $\{a',b,b'\}\setminus\{s_W\}$, which contradicts Claim~\ref{clm:G[U]}.
Therefore $G[U]=P$. 

Every internal vertex of $P$ is in $Y$ since $s_U=u'$.
Hence, by the degree condition of $G$,
 \[
        N_G(x)\cap W=\{a,s_W\}
        \qquad
        \text{for every }x\in V(P)\setminus\{u,u'\},
    \]
and there is no edge of $H$ between $\{a',b,b'\}\setminus\{s_W\}$ and $\{a\}\cup\bigl(V(P)\setminus\{u,u'\}\bigr)$.

Applying the same argument to the edge $u_{t-1}u_t$ instead of $u_0u_1$, we obtain some $s_W'\in \{a,a',b'\}$ such that 
\[
        N_G(x)\cap W=\{b,s_W'\}
        \qquad
        \text{for every }x\in V(P)\setminus\{u,u'\}.
\]
Moreover, there is no edge of $H$ between $\{a,a',b'\}\setminus\{s_W'\}$ and $\{b\}\cup\bigl(V(P)\setminus\{u,u'\}\bigr)$.
Then $s_W=b$ and $s_W'=a$.
Therefore  there is no edge of $H$
between
\[
  \{a',b'\}
  \quad\text{and}\quad
  \{a,b\}\cup\bigl(V(P)\setminus\{u,u'\}\bigr).
\]
Since $G[U]=P$, $H-\{u,u'\}$ is disconnected, which 
contradicts the $3$-connectivity of $H$.
\end{proof}

Recall that $k$ is an odd integer with $k\ge3$. 
We now determine the structure of $G[U]$.

\begin{claim}\label{clm:path-characterization}
\[G[U] \cong \begin{cases}
    P_2 \text{ or } P_3 & \text{ if $k\in\{3,5\}$,} \\
    P_2 & \text{otherwise.}
\end{cases} \]
\end{claim}
\begin{proof}
    Take a nontrivial component $P$ of $G[U]$.
    Such a component $P$ exists by Claim~\ref{clm:edge_exists}. 
    Since $P$ is a path by Claim~\ref{clm:path-components}, it has two end vertices $u$ and $u'$. 
    Then at most two edges of $M$ incident to $z_u$ or $z_{u'}$.
    By Claim~\ref{clm:k3-same-edge} when $k=3$, and since $r\ge 3$ when $k\ge 5$, there exists $ab\in M$ disjoint from $\{z_u,z_{u'}\}$.
    Thus both $a$ and $b$ are adjacent to $u$ and $u'$.
    By Claim~\ref{clm:ends-only},
    \[
    N_G(a)\cap U=N_G(b)\cap U=\{u,u'\}.
    \]
    To show that $G[U]=P$, suppose that $G[U]$ has a component different from $P$. 
    Such a component has a vertex $x$ of degree at most one in $G[U]$. 
    If $x$ is isolated, then, since $d_G(x)\ge k+1=|W|$,  $N_G(x)=W$. 
    Otherwise, $x$ is a leaf and so it has exactly one non-neighbor in $W$ by Claim~\ref{clm:G[U]}. 
    In either case, $x$ has at most one non-neighbor in $W$. 
    Thus $x$ is adjacent to $a$ or $b$, a contradiction. Therefore $G[U]=P$.
    
    We now bound the order of $P$.
    Let $e_u$ and $e_{u'}$ be the edges in $M$ incident to $z_u$ and $z_{u'}$, respectively.
    Set $B:=V(e_u)\cup V(e_{u'})$.
    Claim~\ref{clm:k3-same-edge} gives $|B|= k-1$ when $k=3$, while $|B|\le 4\le k-1$ when $k\ge 5$.
    Hence
    \[ |B|\le \min \{k-1,4\}.\]  
    Moreover, by Claim~\ref{clm:ends-only}, $N_G(x)\cap U=\{u,u'\}$ for every $x\in W\setminus B$.  
    Thus every internal vertex $y$ of $P$, if it exists, satisfies $N_G(y)\cap W \subseteq B$, so 
    \[
    k+1\le d_G(y)=2+|N_G(y)\cap W|\le 2+|B| \le 2+\min\{k-1,4\}.
    \]
    Thus, if $P$ has an internal vertex, then  $k\in \{3,5\}$, $|B|=k-1$, and $N_G(y)\cap W=B$ for every internal vertex $y$ of $P$.

    Suppose that $P$ has at least four vertices.
    Let $x$ and $y$ be two consecutive internal vertices of $P$.
    Since $xy$ is not $k$-removable in $H$, there is a  set $S\subseteq V(H)\setminus\{x,y\}$ with $|S|\le k-1$ such that $xy$ is a cut edge of $H-S$. 
    Since $N_G(x)\cap W=N_G(y)\cap W=B$, $B\subseteq S$.
    Therefore $S=B$.
    Since $|B|=k-1=2(|M|-1)$, there exists $ab \in M$ with $\{a,b\}\cap B=\emptyset$.
    In particular, $a$ is adjacent to both $u$ and $u'$. 
    Then $H-B-xy$ contains a path from $x$ to $y$ as follows: follow $P$ from $x$ to $u$, use the edges $ua$ and $au'$, and follow $P$ from $u'$ to $y$.
    It is a contradiction that $xy$ is a cut edge of $H-B$.
    \end{proof}

    Finally, we derive a contradiction.
    Suppose $G[U]\cong P_2$. 
    Then $|V(G)|=|W|+2=k+3$ and $\delta(G)=k+1=|V(G)|-2$ by Claim~\ref{clm:G[U]}.
    Since $|V(G)|/2=(k+3)/2=m$ and $|V(G)|-3=k$, by Lemma~\ref{lem:dense-matching}, $G$ has a $k$-removable matching of size $m$, a contradiction.
    Suppose that $G[U]\cong P_3$ and $k\in \{3,5\}$.
    Then $|V(G)|\in \{7,9\}$ and $\delta(G)=k+1=|V(G)|-3$.
    Thus by Lemma~\ref{lem:n-3}, $G$ has a $k$-removable matching of size $m$, a contradiction.
\end{proof}

\section{Concluding remarks}

Li, Zhou, Fujita, and Mao~\cite{LZFM26+} introduced the parameter $f(k,\delta)$ as follows. For integers $\delta>k\ge1$, let $f(k,\delta)$ be the largest integer $t$ such that every $k$-connected graph $G$ with $|V(G)|\ge2\delta$ and $\delta(G)\ge\delta$ contains a $k$-removable matching of size $t$. They proved $f(k,\delta)\ge \lceil (\delta +1)/2 \rceil$ when $\delta\ge3k-1$, and obtained bounds for $f(2,\delta)$ and $f(k,k+1)$.
They asked for the exact value of $f(k,\delta)$ in general.
Our main theorem extends their results.
\begin{corollary}\label{cor:f}
For all integers $\delta>k\ge1$ with $(k,\delta)\ne(1,2)$,
\[
    f(k,\delta)\ge
    \left\lceil\frac{\delta+1}{2}\right\rceil.
\]
\end{corollary}

\begin{proof}
Set $m=\lceil(\delta+1)/2\rceil$. Then $2m-2\le\delta$ and $k+1\le\delta$. 
The complete graph excluded in Theorem~\ref{thm:main} has order $2m-1\le\delta+1<2\delta$, so it does not belong to the class defining $f(k,\delta)$. 
Since a cycle has minimum degree $2$, it can belong to the class defining $f(k,\delta)$ only when $\delta=2$.
Therefore Theorem~\ref{thm:main} gives the lower bound $\lceil (\delta +1)/2 \rceil$ unless $(k,\delta)=(1,2)$.
\end{proof}

With the upper bound $f(k,k+1)\le k$ from~\cite{LZFM26+}, Corollary~\ref{cor:f} implies
\[
    \left\lceil\frac{k+2}{2}\right\rceil
    \le f(k,k+1)\le k
\]
for every $k\ge2$. In particular, we determine $f(3,4)=3$.


\bigskip
\noindent 
\textbf{Declaration of generative AI use.}
\par\smallskip
\noindent
During the preparation of this work, the author used OpenAI's GPT-5.6 Sol in order to generate an initial proof of Lemma~\ref{lem:n-3} and to provide grammatical and editorial suggestions on parts of the manuscript. After using this tool, the author independently verified the mathematical arguments, reviewed and revised all AI-assisted text, and takes full responsibility for the correctness of the proof and the content of the published article.

\end{document}